\documentclass[12pt]{article}
\usepackage{pifont}

\usepackage{amsfonts}
\usepackage{latexsym}
\usepackage{amsmath}
\usepackage{amssymb}
\usepackage{color}

\newtheorem{theorem}{Theorem}[section]

\newtheorem{proposition}{Proposition}[section]

\newtheorem{example}{Example}[section]

\newenvironment{proof}[1][Proof]{\noindent \textbf{#1.} }{\ \ \  $\Box$}

\newtheorem{lemma}{Lemma}[section]

\newtheorem{remark}{Remark}[section]

\title{Necessary and sufficient condition for the comparison theorem of multidimensional anticipated backward stochastic differential equations}

\date{}

 \author{ Xiaoming Xu\thanks{E-mail: xmxu@mail.sdu.edu.cn}
 \\ \small{School of Mathematics, Shandong University, Jinan, 250100, China}
 \\ \small{School of Mathematical Sciences, Nanjing Normal University, Nanjing, 210046, China}
 }

\begin{document}

\maketitle

\begin{abstract}
Anticipated backward stochastic differential equations, studied the
first time in 2007, are equations of the following type:
\begin{equation*} \left\{
\begin{tabular}{rlll}
$-dY_t$ &=& $f(t, Y_t, Z_t, Y_{t+\delta(t)},
Z_{t+\zeta(t)})dt-Z_tdB_t, $ & $
t\in[0, T];$\\
$Y_t$ &=& $\xi_t, $ & $t\in[T, T+K];$\\
$Z_t$ &=& $\eta_t, $ & $t\in[T, T+K].$
\end{tabular}\right.
\end{equation*} In this paper, we give a necessary and sufficient
condition under which the comparison theorem holds for
multidimensional anticipated backward stochastic differential
equations with generators independent of the anticipated term of
$Z$.
\\
\par $\textit{Keywords:}$ comparison theorem, multidimensional anticipated backward stochastic
differential equation, necessary and sufficient condition
\end{abstract}



\section{Introduction}

Backward stochastic differential equation (BSDE) of the general form
\begin{equation}\label{equation:BSDE}
Y_t=\xi+\int_t^T g(s, Y_s, Z_s)ds-\int_t^T Z_sdB_s
\end{equation}
was considered the first time by Pardoux-Peng \cite{PP1}. Since
then, the theory of BSDEs has been studied with great interest. One
of the achievements of this theory is the comparison theorem. It is
due to Peng \cite{P} and then generalized by Pardoux-Peng \cite{PP2}
and El Karoui-Peng-Quenez \cite{KPQ}. It allows to compare the
solutions of two BSDEs whenever we can compare the terminal
conditions and the generators. The converse comparison theorem for
BSDEs has also been studied (see [1, 3, 6]). Besides, a necessary
and sufficient condition for the comparison theorem in the
multidimensional case was given by Hu-Peng \cite{HP}, and their main
method consists in translating the comparison principle into an
equivalent viability property for BSDEs, studied by
Buckdahn-Quincampoix-R\v{a}\c{s}canu \cite{BQR}.

Recently, a new type of BSDE, called anticipated BSDE (ABSDE), was
introduced by Peng-Yang \cite{PY} (see also Yang \cite{Y}). The
ABSDE is of the following form:
\begin{equation}\label{equation:PY}
\left\{
\begin{tabular}{lll}
$-dY_t=$&\hspace{-4.5mm} $f(t, Y_t, Z_t, Y_{t+\delta(t)},
Z_{t+\zeta(t)})dt-Z_tdB_t, $ & $
t\in[0, T];$\\
$Y_t=$ &\hspace{-8.5mm} $\xi_t, $ & $t\in[T, T+K];$\\
$Z_t=$ &\hspace{-8.5mm} $\eta_t, $ & $t\in[T, T+K],$
\end{tabular}\right.
\end{equation}
where $\delta(\cdot): [0, T]\rightarrow \mathbb{R^+} \setminus
\{0\}$ and $\zeta(\cdot): [0, T]\rightarrow \mathbb{R^+} \setminus
\{0\}$ are continuous functions satisfying

$\mathbf{(a1)}$ there exists a constant $K \geq 0$ such that for
each $t\in[0, T],$
$$t+\delta(t) \leq T+K,\quad t+\zeta(t) \leq T+K;$$

$\mathbf{(a2)}$ there exists a constant $M \geq 0$ such that for
each $t\in[0, T]$ and each nonnegative integrable function
$g(\cdot)$, $$\int_t^T g(s+\delta(s))ds\leq M\int_t^{T+K}
g(s)ds,\quad \int_t^T g(s+\zeta(s))ds\leq M\int_t^{T+K} g(s)ds.$$

\cite{PY} tells us that (\ref{equation:PY}) has a unique solution
under proper assumptions. Furthermore, for $1$-dimensional ABSDEs
there is a comparison theorem, which requires that the generators of
the ABSDEs cannot depend on the anticipated term of $Z$ and one of
them must be increasing in the anticipated term of $Y$.

The aim of this paper is to give a comparison theorem for
multidimensional ABSDEs with generators independent of the
anticipated term of $Z$ and possibly not increasing in the
anticipated term of $Y$. Moreover, the condition under which the
comparison theorem holds is necessary and sufficient. The main
approach we adopt is to consider an ABSDE as a series of BSDEs and
then apply the results in \cite{HP}. It should be mentioned here
that the reason why the generators are still required to be
independent of the anticipated term of $Z$ is that the continuity
property of $f(\cdot, y, z, Y_{\cdot+\delta(\cdot)},
Z_{\cdot+\zeta(\cdot)})$, where $(Y, Z)$ is the unique solution to a
BSDE, is hard to depict.

The paper is organized as follows: in Section $2$, we list some
notations and some existing results which will be used in the text.
In Section $3$, we mainly study the comparison theorem for
multidimensional ABSDEs, besides, we also discuss a lot about that
for $1$-dimensional ABSDEs.

\section{Preliminaries}

Let $\{B_t; t\geq 0\}$ be a $d$-dimensional standard Brownian motion
on a probability space $(\Omega, \mathcal{F}, P)$ and
$\{\mathcal{F}_t; t\geq 0\}$ be its natural filtration. Denote by
$|\cdot|$ the norm in $\mathbb{R}^m.$ Given $T
>0,$ we will use the following notations:

$\bullet$ $L^2(\mathcal{F}_T; \mathbb{R}^m)$ := $\{\xi\in
\mathbb{R}^m$ $|$ $\xi$ is an $\mathcal{F}_T$-measurable random
variable such that $E|\xi|^2< + \infty\};$

$\bullet$ $L_{\mathcal{F}}^2(0, T; \mathbb{R}^m)$ := $\{ \varphi:
\Omega\times [0, T]\rightarrow \mathbb{R}^m$ $|$ $\varphi$ is
progressively measurable and $E\int_0^T |\varphi_t|^2dt< +
\infty\};$

$\bullet$ $S_{\mathcal{F}}^2(0, T; \mathbb{R}^m)$ := $\{\psi:
\Omega\times [0, T]\rightarrow \mathbb{R}^m$ $|$ $\psi$ is
progressively measurable and $E[\sup_{0 \leq t \leq T} |\psi_t|^2]<
+ \infty\}.$

\subsection{Comparison theorem for multidimensional BSDEs}

Consider the BSDE (\ref{equation:BSDE}). For the generator $g:
\Omega\times [0, T]\times \mathbb{R}^m\times \mathbb{R}^{m\times
d}\rightarrow \mathbb{R}^m$, we make the following assumptions,
which are essential for \cite{BQR} as well as \cite{HP}:

${\bf{(A1)}}$ $g(\cdot, \cdot, y, z)$ is progressively measurable,
and for each $(y, z),$ $g(\omega, \cdot, y, z)$ is continuous, a.s.;

${\bf{(A2)}}$ there exists a constant $L_g \geq 0$ such that for
each $s\in [0, T],$ $y, y^\prime \in \mathbb{R}^m,$ $z, z^\prime \in
\mathbb{R}^{m\times d},$ the following holds:
$$|g(s, y, z)-g(s, y^\prime, z^\prime)|\leq L_g(|y- y^\prime|+|z-z^\prime|);$$

${\bf{(A3)}}$ $\sup_{0\leq s\leq T} |g(s, 0, 0)|\in
L^2(\mathcal{F}_T; \mathbb{R}).$

Then according to \cite{PP1}, for each $\xi\in L^2(\mathcal{F}_T;
\mathbb{R}^m),$ BSDE (\ref{equation:BSDE}) has a unique solution.

We recall here the comparison theorem for multidimensional BSDEs
from \cite{HP}. For $j=1, 2$, let $(Y^{(j)}, Z^{(j)})$ be the unique
solution to the following BSDE:
\begin{equation}\label{equation:hu peng comparison}
Y_t^{j}=\xi^j + \int_t^T g_j(s, Y_s^{j}, Z_s^{j})ds- \int_t^T
Z_s^{j}dB_s,
\end{equation}
where $\xi^j\in L^2(\mathcal{F}_T; \mathbb{R}^m)$ and $g_j$
satisfies (A1)--(A3).

\begin{theorem}
{\it The following are equivalent:

{\rm (i)} for all $\tau\in [0, T],$  $\xi^j\in L^2(\mathcal{F}_\tau;
\mathbb{R}^m)$ $(j=1, 2)$ such that $\xi^1 \geq \xi^2$, the unique
solutions $(Y^{j}, Z^{j})\in S_\mathcal{F}^2(0, \tau;
\mathbb{R}^m)\times L_\mathcal{F}^2(0, \tau; \mathbb{R}^{m\times
d})$ $(j=1, 2)$ to the BSDE {\rm(\ref{equation:hu peng comparison})}
over time interval $[0, \tau]\!:$
\begin{equation*}
Y_t^{j}=\xi^{j}+\int_t^\tau g_j(s, Y_s^{j}, Z_s^{j})ds-\int_t^\tau
Z_s^{j}dB_s,
\end{equation*}
satisfy
\begin{equation*}
Y_t^{1}\geq Y_t^{2}, {\rm{\ for\ all}}\ t\in[0, \tau],\ {\rm{a.s.}};
\end{equation*}

{\rm (ii)} for all $t\in [0, T]$, $(y, z), (y^\prime, z^\prime)\in
\mathbb{R}^m\times \mathbb{R}^{m\times d},$
\begin{equation}\label{equation:hu peng bsvp}
-4\langle y^-, g_1(t, y^++y^\prime, z)-g_2(t, y^\prime,
z^\prime)\rangle\leq 2\sum_{k=1}^m
\mathbf{1}_{\{y_k<0\}}|z_k-z_k^\prime|^2+C|y^-|^2,\ {\rm{a.s.}},
\end{equation}
where $C>0$ is a constant.}
\end{theorem}

\begin{remark} In fact, the constant $C$ in (\ref{equation:hu
peng bsvp}) only depends on the Lipschitz coefficients $L_{g_j}$
$(j=1, 2)$ of the generators $g_j$ $(j=1, 2)$, which can be easily
got from the detailed proofs in \cite{BQR} and \cite{HP}.
\end{remark}

\begin{remark} Let $m=1$. Then (\ref{equation:hu peng bsvp}) is
equivalent to
\begin{equation*}
g_1(t, y, z)\geq g_2(t, y, z).
\end{equation*}
This has been stated already in \cite{HP}.
\end{remark}

\subsection{Multidimensional anticipated BSDEs}

Now let us consider the ABSDE (\ref{equation:PY}). First for the
generator $f(\omega, s, y, z, \theta, \phi): \Omega \times [0,
T]\times \mathbb{R}^m\times \mathbb{R}^{m\times d}\times
S_\mathcal{F}^2(s, T+K; \mathbb{R}^m)\times L_\mathcal{F}^2(s, T+K;
\mathbb{R}^{m\times d})\rightarrow L^2 (\mathcal{F}_s;
\mathbb{R}^m),$ we introduce two hypotheses:

${\bf{(H1)}}$ there exists a constant $L_f > 0$ such that for each
$s\in [0, T],$ $y, y^\prime\in \mathbb{R}^m,$ $z, z^\prime \in
\mathbb{R}^{m\times d},$ $\theta, \theta^\prime \in
L_\mathcal{F}^2(s, T+K; \mathbb{R}^m),$ $\phi, \phi^\prime \in
L_\mathcal{F}^2(s, T+K; \mathbb{R}^{m\times d}),$ $r, \bar{r}\in [s,
T+K],$ the following holds:
\begin{equation*}
|f(s, y, z, \theta_r, \phi_{\bar{r}})-f(s, y^\prime, z^\prime,
\theta_r^\prime, \phi_{\bar{r}}^\prime)|\leq L_f(|y-
y^\prime|+|z-z^\prime|+E^{\mathcal{F}_s}[|\theta_r-\theta_r^\prime|+|\phi_{\bar{r}}-\phi_{\bar{r}}^\prime|]);
\end{equation*}

${\bf{(H2)}}$ $E[\int_0^T |f(s, 0, 0, 0, 0)|^2ds]< +\infty.$

Let us review the existence and uniqueness theorem for ABSDEs from
\cite{PY}:

\begin{theorem}
{\it Assume that $f$ satisfies {\rm(H1)} and
{\rm(H2)}, $\delta,$ $\zeta$ satisfy {\rm(a1)} and {\rm(a2)}, then
for arbitrary given terminal conditions $(\xi, \eta) \in
S_\mathcal{F}^2(T, T+K; \mathbb{R}^m)\times L_\mathcal{F}^2(T, T+K;
\mathbb{R}^{m\times d}),$ the ABSDE {\rm(\ref{equation:PY})} has a
unique solution, i.e., there exists a unique pair of processes $(Y,
Z)\in S_\mathcal{F}^2(0, T+K; \mathbb{R}^m)\times L_\mathcal{F}^2(0,
T+K; \mathbb{R}^{m\times d})$ satisfying {\rm(\ref{equation:PY})}.
}
\end{theorem}
\vspace{1mm}

Next we will recall the comparison theorem from Peng-Yang \cite{PY}.
For $j=1, 2$, let $(Y^{(j)}, Z^{(j)})$ be the unique solution to the
following $1$-dimensional ABSDE:
\begin{equation}\label{equation:peng yang comparison}
\left\{
\begin{tabular}{llll}
$-dY_t^{(j)}$ &\hspace{-4mm}=&\hspace{-4mm} $f_j(t, Y_t^{(j)},
Z_t^{(j)}, Y_{t+\delta(t)}^{(j)})dt-Z_t^{(j)}dB_t, $ & $
t\in[0, T];$\\
$Y_t^{(j)}$ &\hspace{-4mm}=&\hspace{-4mm} $\xi_t^{(j)}, $ & $t\in[T,
T+K].$
\end{tabular}\right.
\end{equation}

\begin{theorem}
{\it Assume that $f_1,$ $f_2$ satisfy {\rm(H1)}
and {\rm(H2)}, $\xi^{(1)}, \xi^{(2)} \in S_\mathcal{F}^2(T, T+K;
\mathbb{R}),$ $\delta$ satisfies {\rm(a1)}, {\rm(a2)}, and for each
$t\in [0, T],$ $y\in \mathbb{R},$ $z\in \mathbb{R}^d,$ $f_2(t, y, z,
\cdot)$ is increasing, i.e., $f_2(t, y, z, \theta_r)\geq f_2(t, y,
z, \theta_r^\prime)$, if $\theta_r\geq \theta_r^\prime,$ $\theta,
\theta^\prime\in L_\mathcal{F}^2(t, T+K; \mathbb{R}), r\in [t,
T+K].$ If $\xi_s^{(1)}\geq \xi_s^{(2)}, s\in [T, T+K]$ and $f_1(t,
y, z, \theta_r)\geq f_2(t, y, z, \theta_r), t\in[0, T], y\in
\mathbb{R}, z\in \mathbb{R}^d, \theta\in L_\mathcal{F}^2(t, T+K;
\mathbb{R}), r\in [t, T+K],$ then $Y_t^{(1)}\geq Y_t^{(2)},\
{\rm{a.e.}},\ {\rm{a.s.}}$ }
\end{theorem}

\vspace{1mm}

At the end of this subsection, for $f(\omega, s, y, z, \theta):
\Omega \times [0, T]\times \mathbb{R}^m\times \mathbb{R}^{m\times
d}\times S_\mathcal{F}^2(s, T+K; \mathbb{R}^m)\rightarrow L^2
(\mathcal{F}_s; \mathbb{R}^m)$ particularly, let us introduce three
more hypotheses:

${\bf{(H3)}}$ for each $(y, z, \theta)\in \mathbb{R}^m\times
\mathbb{R}^{m\times d} \times S_\mathcal{F}^2(\cdot, T+K;
\mathbb{R}^m)$, $f(\omega, \cdot, y, z,
\theta_{\cdot+\delta(\cdot)})$ is continuous, a.s.;

${\bf{(H4)}}$ $\sup_{0\leq s\leq T} |f(s, 0, 0, 0)|\in
L^2(\mathcal{F}_T; \mathbb{R});$

${\bf{(H4^\prime)}}$ $\sup_{0\leq s\leq T} |f(s, 0, 0,
\theta_{s+\delta(s)})|\in L^2(\mathcal{F}_T; \mathbb{R}),$ for all
$\theta \in S_{\mathcal{F}}^2(s, T+K; \mathbb{R}^m).$

\begin{remark}
$(H4')$ $\Rightarrow$ $(H4)$ $\Rightarrow$ $(H2)$;
$(H1)$ + $(H4)$ $\Rightarrow$ $(H4^\prime)$. Indeed, we have
\begin{align*}
E&\Big[\sup_{0\leq s\leq T}|f(s, 0, 0, \theta_{s+\delta(s)})|^2\Big] \\
&\leq 2E\Big[\sup_{0\leq s\leq T}|f(s, 0, 0,
\theta_{s+\delta(s)})-f(s, 0, 0, 0)|^2+\sup_{0\leq s\leq T}|f(s, 0,
0, 0)|^2\Big]
\\
&\leq 2L_f^2 E\Big[\sup_{0\leq s\leq
T}E^{\mathcal{F}_s}|\theta_{s+\delta(s)}|^2\Big]+2E \Big[\sup_{0\leq
s\leq
T}|f(s, 0, 0, 0)|^2\Big] \\
&\leq 2L_f^2 E\Big[\sup_{0\leq s\leq
T}|\theta_{s+\delta(s)}|^2\Big]+2E \Big[\sup_{0\leq s\leq
T}|f(s, 0, 0, 0)|^2\Big]\\
&< + \infty.
\end{align*}
\end{remark}

\section{Comparison theorem for anticipated BSDEs}

Consider the following multidimensional ABSDE:
\begin{equation}\label{equation:conparison new}
\left\{
\begin{tabular}{llll}
$-dY_t^{(j)}$ &\hspace{-4mm}=&\hspace{-4mm} $f_j(t, Y_t^{(j)},
Z_t^{(j)}, Y_{t+\delta^{(j)}(t)}^{(j)})dt-Z_t^{(j)}dB_t, $ & $
t\in[0, T];$\\
$Y_t^{(j)}$ &\hspace{-7.5mm}=&\hspace{-7mm} $\xi_t^{(j)}, $ & $t\in
[T, T+K],$
\end{tabular}\right.
\end{equation}
where $j=1, 2$, $f_j$ satisfies (H1), (H3) and (H4), $\xi^{(j)} \in
S_\mathcal{F}^2(T, T+K; \mathbb{R}^m)$, $\delta^{(j)}$ satisfies
(a1) and (a2). Then by Theorem 2.2, (\ref{equation:conparison new})
has a unique solution.

\begin{proposition}
{\it Putting $t_0=T$, we define by iteration
\begin{equation*}
t_i:=\min\{t \in [0, T]: \min\{s+ \delta^{(1)}(s),\
s+\delta^{(2)}(s)\}\geq t_{i-1}, {\rm{\ for\ all}}\ s\in [t, T]\},\
\ \ i\geq 1.
\end{equation*}
Set $N:=\max\{i: t_{i-1}>0\}$. Then $N$ is finite, $t_N=0$ and
\begin{equation*}
[0, T]=[0, t_{N-1}]\cup [t_{N-1}, t_{N-2}]\cup \cdots \cup [t_2,
t_1]\cup [t_1, T].
\end{equation*}}
\end{proposition}

\vspace{-5mm}

\begin{proof} Let us first prove that $N$ is finite. For this
purpose, we apply the method of reduction to absurdity. Suppose that
$N$ is infinite. From the definition of $\{t_i\}_{i=1}^{+ \infty}$,
we know
\begin{equation}\label{equation:series t}
\min\{t_i+ \delta^{(1)}(t_i),\ t_i+ \delta^{(2)}(t_i)\}=t_{i-1},\
i=1, 2,\ldots.
\end{equation}
Since $\delta^{(j)}(\cdot)$ $(j=1, 2)$ are continuous and positive,
thus obviously we have $$t_i<t_{i-1},\ \ i=1, 2,\ldots.$$ Therefore
$\{t_i\}_{i=1}^{+ \infty}$ converges as a strictly monotone and
bounded sequence. Denote its limit by $\bar{t}.$ Letting
$i\rightarrow + \infty$ on both sides of (\ref{equation:series t}),
we get
$$\min\{\bar{t}+\delta^{(1)}(\bar{t}),\ \bar{t}+\delta^{(2)}(\bar{t})\}=\bar{t}.$$ Hence
$\delta^{(1)}(\bar{t})=0$ or $\delta^{(2)}(\bar{t})=0$, which is
just a contradiction since both $\delta^{(1)}$ and $\delta^{(2)}$
are positive. Consequently, $N$ is finite.

Next we will show that $t_N=0.$ In fact, the following holds
obviously:
$$\min\{t_N+\delta^{(1)}(t_N),\ t_N+\delta^{(2)}(t_N)\}> t_N,$$ which implies
$t_N = 0,$ or else  we can find a $\tilde{t}\in [0, t_N)$ due to the
continuity of $\delta^{(j)}(\cdot)$ $(j=1, 2)$ such that
$$\min\{s+\delta^{(1)}(s),\ s+\delta^{(2)}(s)\}\geq t_N,\quad {\rm{\ for\ all}}\ s\in [\tilde{t}, T],$$ from which we know that $\tilde{t}$
is an element of the sequence as well.
\end{proof}

\begin{proposition}
{\it For $j=1, 2$, suppose that $(Y^{(j)},
Z^{(j)})$ is the unique solution to the ABSDE
{\rm(\ref{equation:conparison new})}. Then for fixed $i\in {\{1,
2,\ldots, N\}},$ $\tau \in [t_i, t_{i-1}],$ over time interval
$[t_i, \tau]$, ABSDE {\rm(\ref{equation:conparison new})} is
equivalent to the following ABSDE{\,\rm:}
\begin{equation}\label{equation: in prop 2}
\left\{
\begin{array}{llll}
-d\bar{Y}_t^{(j)} &\hspace{-2.5mm}= f_j(t, \bar{Y}_t^{(j)},
\bar{Z}_t^{(j)},
\bar{Y}_{t+\delta^{(j)}(t)}^{(j)})dt-\bar{Z}_t^{(j)}dB_t,  &
t\in[t_i, \tau];\\
\bar{Y}_t^{(j)} &\hspace{-7.4mm}= Y_t^{(j)},  & t\in [\tau, T+K],
\end{array}\right.
\end{equation}
which is also equivalent to the following BSDE with terminal
condition $Y_\tau^{(j)}\!:$
\begin{equation}\label{equation: in prop 3}
\tilde{Y}_t^{(j)}=Y_\tau^{(j)}+\int_t^\tau f_j(s, \tilde{Y}_s^{(j)},
\tilde{Z}_s^{(j)}, Y_{s+\delta^{(j)}(s)}^{(j)})ds-\int_t^\tau
\tilde{Z}_s^{(j)}dB_s.
\end{equation}
That is to say,
$$
Y_t^{(j)}=\bar{Y}_t^{(j)}=\tilde{Y}_t^{(j)},\
Z_t^{(j)}=\bar{Z}_t^{(j)}=\tilde{Z}_t^{(j)},\quad t\in [t_i,
\tau],\quad j=1,2.$$}
\end{proposition}

\begin{proof} The conclusion immediately follows from

(i) for each $s\in [t_i, \tau]$, $s+\delta^{(j)}(s)\geq t_{i-1}$;

(ii) write $f_j^Y(s, y, z)=f_j(s, y, z,
Y_{s+\delta^{(j)}(s)}^{(j)})$, then $f_j^Y$ satisfies (A1)--(A3);

(iii) $(Y_t^{(j)}, Z_t^{(j)})_{t\in [t_i, \tau]}$ satisfies both
ABSDE (\ref{equation: in prop 2}) and BSDE (\ref{equation: in prop
3}).
\end{proof}

\subsection{Comparison theorem in $\mathbb{R}^m$}

Next we will study the following problem: under which condition the
comparison theorem for multidimensional ABSDEs holds?

\begin{lemma}
{\it For fixed $i\in {\{1, 2,\ldots, N\}},$ $s\in
(t_i, t_{i-1}),$ the following are equivalent:

{\rm(i)} for all $\tau\in [t_i, s],$  $\xi^{(j)}\in
S_\mathcal{F}^2(\tau, T+K; \mathbb{R}^m)$ $(j=1, 2)$ such that
$\xi^{(1)}\geq \xi^{(2)}$ and $(\xi_t^{(j)})_{t\in [t_{i-1}, T+K]}$
$(j=1, 2)$ are fixed, the unique solutions $(Y^{(j)}, Z^{(j)})\in
S_\mathcal{F}^2(t_i, T+K; \mathbb{R}^m)\times L_\mathcal{F}^2(t_i,
\tau; \mathbb{R}^{m\times d})$ $(j=1, 2)$ to the following ABSDE
over time interval $[t_i, T+K]:$
\begin{equation}\label{equation:conparison new tau in lemma}
\left\{
\begin{array}{llll}
-dY_t^{(j)}&\hspace{-3mm}=&\hspace{-3mm} f_j(t, Y_t^{(j)},
Z_t^{(j)}, Y_{t+\delta^{(j)}(t)}^{(j)})dt-Z_t^{(j)}dB_t, &
t\in[t_i, \tau];\\
Y_t^{(j)} &\hspace{-7.5mm}=&\hspace{-6.5mm} \xi_t^{(j)},  & t\in
[\tau, T+K],
\end{array}\right.
\end{equation}
satisfy \begin{align*} Y_t^{(1)}\geq Y_t^{(2)}, {\rm{\ for\ all}}\
t\in[t_i, \tau],\ {\rm{a.s.}};\end{align*}

{\rm(ii)} for all $t\in [t_i, s]$, $(y, z), (y^\prime, z^\prime)\in
\mathbb{R}^m\times \mathbb{R}^{m\times d},$
\begin{equation*}
-4\langle y^-, f_1(t, y^++y^\prime, z,
\xi_{t+\delta^{(1)}(t)}^{(1)})-f_2(t, y^\prime, z^\prime,
\xi_{t+\delta^{(2)}(t)}^{(2)})\rangle \leq 2\sum_{k=1}^m
\mathbf{1}_{\{y_k<0\}}|z_k-z_k^\prime|^2+C|y^-|^2,\ \ {\rm{a.s.}},
\end{equation*}
where $C>0$ is a constant. }
\end{lemma}

\begin{proof} According to Proposition 3.2, we can equivalently
consider the following BSDE over time interval $[t_i, \tau]$ instead
of (\ref{equation:conparison new tau in lemma}):
\begin{equation}\label{equation:BSDE in lemma}
\tilde{Y}_t^{(j)}=\xi_\tau^{(j)}+\int_t^\tau f_j(r,
\tilde{Y}_r^{(j)}, \tilde{Z}_r^{(j)},
\xi_{r+\delta^{(j)}(r)}^{(j)})dr-\int_t^\tau \tilde{Z}_r^{(j)}dB_r.
\end{equation}
Write $f_j^\xi(s, y, z)=f_j(s, y, z,
\xi_{s+\delta^{(j)}(s)}^{(j)})$, then $f_j^\xi$ satisfies
(A1)--(A3).

On the other hand, it is obvious that (i) is equivalent to

(iii) for all $\tau\in [t_i, s],$ $\xi_\tau^{(j)}\in
L^2(\mathcal{F}_\tau; \mathbb{R}^m)$ $(j=1, 2)$ such that
$\xi_\tau^{(1)}\geq \xi_\tau^{(2)},$ the unique solutions
$(\tilde{Y}^{(j)}, \tilde{Z}^{(j)})\in S_\mathcal{F}^2(t_i, \tau;
\mathbb{R}^m)\times L_\mathcal{F}^2(t_i, \tau; \mathbb{R}^{m\times
d})$ $(j=1, 2)$ to the BSDE (\ref{equation:BSDE in lemma}) satisfy
$$
\tilde{Y}_t^{(1)}\geq \tilde{Y}_t^{(2)}, {\rm{\ for\ all}}\
t\in[t_i, \tau],\ {\rm{a.s.}}
$$

By Theorem 2.1, (iii) is equivalent to (ii).
\end{proof}

\begin{remark}
From the proof of Lemma 3.1, we can find that the result holds true
for arbitrary values of $(\xi^{(j)})_{t\in (\tau, t_{i-1})}$ such
that $\xi^{(j)}\in S_\mathcal{F}^2(\tau, T+K; \mathbb{R}^m)$. In
fact we even can choose them according to a fixed formula, for
example, all $\tau\in [t_i, s],$ for all $\xi_\tau^{(j)}\in
L^2(\mathcal{F}_\tau; \mathbb{R}^m)$ $(j=1,2)$ such that
$\xi_\tau^{(1)}\geq \xi_\tau^{(2)},$ and the fixed processes
$(\xi_t^{(j)})_{t\in [t_{i-1}, T+K]}$ $(j=1, 2)$ such that
$\xi_t^{(1)}\geq \xi_t^{(2)}$, we can construct ${\xi}^{(j)}\in
S_\mathcal{F}^2(\tau, t_{i-1}; \mathbb{R}^m)$ $(j=1, 2)$, thanks to
the strict inequality $\tau < t_{i-1}$, such that ${\xi}^{(1)}\geq
{\xi}^{(2)}$ as follows:
$$
{\xi}_t^{(j)}:=\frac{t_{i-1}-t}{t_{i-1}-\tau}\xi_\tau^{(j)}+\frac{t-\tau}{t_{i-1}-\tau}E^{\mathcal{F}_t}[\xi_{t_{i-1}}^{(j)}],\quad
t\in [\tau, t_{i-1}].
$$
\end{remark}

\begin{theorem}
{\it The following are equivalent:

{\rm(i)} for all $\tau\in [0, T],$ $\xi^{(j)}\in
S_\mathcal{F}^2(\tau, T+K; \mathbb{R}^m)$ $(j=1, 2)$ such that
$\xi^{(1)}\geq \xi^{(2)},$ the unique solutions $(Y^{(j)},
Z^{(j)})\in S_\mathcal{F}^2(0, T+K; \mathbb{R}^m)\times
L_\mathcal{F}^2(0, \tau; \mathbb{R}^{m\times d})$ $(j=1, 2)$ to the
following ABSDE:
\begin{equation}\label{equation:conparison new tau}
\left\{
\begin{array}{llll}
-dY_t^{(j)} &\hspace{-3mm}=&\hspace{-2.5mm} f_j(t, Y_t^{(j)},
Z_t^{(j)}, Y_{t+\delta^{(j)}(t)}^{(j)})dt-Z_t^{(j)}dB_t,  &
t\in[0, \tau];\\
Y_t^{(j)} &\hspace{-7mm}=&\hspace{-6.5mm} \xi_t^{(j)},  & t\in
[\tau, T+K],
\end{array}\right.
\end{equation}
satisfy
\begin{align*}
Y_t^{(1)}\geq Y_t^{(2)}, {\rm{\ for\ all}}\ t\in[0, \tau],\quad
{\rm{a.s.}};
\end{align*}

{\rm(ii)} for all $s\in [0, T]$, $(y, z), (y^\prime, z^\prime)\in
\mathbb{R}^m\times \mathbb{R}^{m\times d}$ and all $\theta^{(j)}\in
S_\mathcal{F}^2(s, T+K; \mathbb{R}^m)$ $(j=1, 2)$ such that
$\theta^{(1)} \geq \theta^{(2)},$
\begin{equation}\label{equation:BSVP in comparison thm 3}
-4\langle y^-, f_1(s, y^++y^\prime, z,
\theta_{s+\delta^{(1)}(s)}^{(1)})-f_2(s, y^\prime, z^\prime,
\theta_{s+\delta^{(2)}(s)}^{(2)})\rangle \leq 2\sum_{k=1}^m
\mathbf{1}_{\{y_k<0\}}|z_k-z_k^\prime|^2+C|y^-|^2,\ \ {\rm{a.s.}},
\end{equation}
where $C>0$ is a constant. }
\end{theorem}

\begin{proof} (a) (i) $\Rightarrow$ (ii): without loss of
generality, we may assume that $s\in [t_i, t_{i-1}]$ $(i\in {\{1,
2,\ldots, N\}}).$ For some convenience of techniques, we first
consider the case when $s\in (t_i, t_{i-1}).$

According to Proposition 3.2, (i) implies

(iii) for all $\tau\in [t_i, s],$ the unique solutions $(Y^{(j)},
Z^{(j)})\in S_\mathcal{F}^2(t_i, T+K; \mathbb{R}^m)\times
L_\mathcal{F}^2(t_i, \tau; \mathbb{R}^{m\times d})$ $(j=1, 2)$ to
the following ABSDE over time interval $[t_i, T+K]\!:$
\begin{equation*}
\left\{
\begin{array}{llll}
-dY_t^{(j)} &\hspace{-2mm}=& \hspace{-2mm}f_j(t, Y_t^{(j)},
Z_t^{(j)}, Y_{t+\delta^{(j)}(t)}^{(j)})dt-Z_t^{(j)}dB_t,  &
t\in[t_i, \tau];\\
Y_t^{(j)} &\hspace{-7mm}=&\hspace{-7mm} {\xi}_t^{(j)},  & t\in
[\tau, T+K],
\end{array}\right.
\end{equation*}
satisfy $$Y_t^{(1)}\geq Y_t^{(2)}, {\rm{\ for\ all}}\ t\in[t_i,
\tau],\ {\rm{a.s.}}$$

In the above ABSDE, let $$(\xi_t^{(j)})_{t\in [t_{i-1},
T+K]}=(\theta_t^{(j)})_{t\in [t_{i-1}, T+K]}.$$

Then by Lemma 3.1, (iii) is equivalent to

(iv) for all $t\in [t_i, s]$, $(y, z), (y^\prime, z^\prime)\in
\mathbb{R}^m\times \mathbb{R}^{m\times d},$
\begin{equation}\label{equation:BSVP r:=s}
-4\langle y^-, f_1(t, y^++y^\prime, z,
\theta_{t+\delta^{(1)}(t)}^{(1)})-f_2(t, y^\prime, z^\prime,
\theta_{t+\delta^{(2)}(t)}^{(2)})\rangle \leq 2\sum_{k=1}^m
\mathbf{1}_{\{y_k<0\}}|z_k-z_k^\prime|^2+C|y^-|^2,\ \ {\rm{a.s.}}
\end{equation}

Setting $t=s$ in (\ref{equation:BSVP r:=s}), we can get
(\ref{equation:BSVP in comparison thm 3}) for $s\in (t_i, t_{i-1})$.
Note the continuity property of $f_j$ $(j=1, 2)$, then
(\ref{equation:BSVP in comparison thm 3}) holds for each $s\in [t_i,
t_{i-1}]$.

(b) (ii) $\Rightarrow$ (i): we only need to consider the nontrivial
case $\tau \neq 0.$ Without loss of generality, we may assume that
$\tau \in (t_i, t_{i-1}]$ $(i\in {\{1, 2,\ldots, N}\}).$ Our aim is
to show that
\begin{equation*}
Y_t^{(1)}\geq Y_t^{(2)}, {\rm{\ for\ all}}\ t\in[0, \tau],\
{\rm{a.s.}},
\end{equation*}
where $Y^{(j)}$ $(j=1, 2)$ are the unique solutions of ABSDE
(\ref{equation:conparison new tau}).

Consider the ABSDE (\ref{equation:conparison new tau}) one time
interval by one time interval.

For the first step, we consider the case when $t\in [t_i, \tau]$.
According to Proposition 3.2, we can equivalently consider the
following BSDE instead of ABSDE (\ref{equation:conparison new tau}):
\begin{equation*}
\tilde{Y}_t^{(j)}=\xi_\tau^{(j)}+\int_t^\tau f_j(s,
\tilde{Y}_s^{(j)}, \tilde{Z}_s^{(j)},
\xi_{s+\delta^{(j)}(s)}^{(j)})ds-\int_t^\tau \tilde{Z}_s^{(j)}dB_s.
\end{equation*}
Noticing that $\xi^{(j)}\in S_\mathcal{F}^2(\tau, T+K;
\mathbb{R}^m)$ $(j=1, 2)$ and $\xi^{(1)}\geq \xi^{(2)},$ from (ii),
we get

(v) for all $s\in [t_i, \tau]$, $(y, z), (y^\prime, z^\prime) \in
\mathbb{R}^m\times \mathbb{R}^{m\times d},$
\begin{equation*}
-4\langle y^-, f_1(s, y^++y^\prime, z,
\xi_{s+\delta^{(1)}(s)}^{(1)})-f_2(s, y^\prime, z^\prime,
\xi_{s+\delta^{(2)}(s)}^{(2)})\rangle \leq 2\sum_{k=1}^m
\mathbf{1}_{\{y_k<0\}}|z_k-z_k^\prime|^2+C|y^-|^2,\ \ {\rm{a.s.}}
\end{equation*}
Then thanks to Theorem 2.1, (v) implies
\begin{equation*}
\tilde{Y}_t^{(1)}\geq \tilde{Y}_t^{(2)}, {\rm{\ for\ all}}\
t\in[t_i, \tau],\ {\rm{a.s.}},
\end{equation*}
i.e.,
\begin{equation*}
Y_t^{(1)}\geq Y_t^{(2)}, {\rm{\ for\ all}}\ t\in[t_i, \tau],\
{\rm{a.s.}}
\end{equation*}
Consequently,
\begin{equation}\label{equation:II Y1 geq Y2 ti tau}
Y_t^{(1)}\geq Y_t^{(2)}, {\rm{\ for\ all}}\ t\in[t_i, T+K],\
{\rm{a.s.}}
\end{equation}

For the second step, we consider the case when $t\in [t_{i+1},
t_i]$. Similarly, according to Proposition 3.2, we can consider the
following BSDE equivalently:
\begin{equation*}
\tilde{Y}_t^{(j)}=Y_{t_i}^{(j)}+\int_t^{t_i} f_j(s,
\tilde{Y}_s^{(j)}, \tilde{Z}_s^{(j)},
Y_{s+\delta^{(j)}(s)}^{(j)})ds-\int_t^{t_i} \tilde{Z}_s^{(j)}dB_s.
\end{equation*}
Noticing (\ref{equation:II Y1 geq Y2 ti tau}), according to (ii), we
have

(vi) for all $s\in [t_{i+1}, t_i]$, $(y, z), (y^\prime, z^\prime)
\in \mathbb{R}^m\times \mathbb{R}^{m\times d},$
\begin{equation*}
-4\langle y^-, f_1(s, y^++y^\prime, z,
Y_{s+\delta^{(1)}(s)}^{(1)})-f_2(s, y^\prime, z^\prime,
Y_{s+\delta^{(2)}(s)}^{(2)})\rangle \leq 2\sum_{k=1}^m
\mathbf{1}_{\{y_k<0\}}|z_k-z_k^\prime|^2+C|y^-|^2,\ \ {\rm{a.s.}}
\end{equation*}
Applying Theorem 2.1 again, from (vi), we can finally get
$$Y_t^{(1)}\geq Y_t^{(2)}, {\rm{\ for\ all}}\ t\in[t_{i+1}, t_i],\ {\rm{a.s.}}$$

Similarly to the above steps, we can give the proofs for the other
cases when $t\in [t_{i+2}, t_{i+1}],[t_{i+3}, t_{i+2}],$
$\ldots,[t_N, t_{N-1}].$
\end{proof}

\begin{remark}
By Remark 2.1, we can deduce from the fact
\begin{equation*}
L_{f^\theta}=L_{f^{\theta^\prime}}, {\rm{\ for\ all}}\ \theta,
\theta^\prime \in S_\mathcal{F}^2(s, T+K; \mathbb{R}^m)
\end{equation*}
where $f^\theta(s, y, z)=f(s, y, z, \theta_{s+\delta(s)})$ and
$f^{\theta^\prime}(s, y, z)=f(s, y, z,
\theta_{s+\delta(s)}^\prime)$, that the constant $C$ in
(\ref{equation:BSVP in comparison thm 3}) is independent of
$\theta^{(j)}$ $(j=1, 2)$ and only depends on the Lipschitz
coefficients of $f_j$ $(j=1, 2)$.
\end{remark}

\begin{remark}
If $\delta^{(1)}=\delta^{(2)}=:\delta,$ then
(\ref{equation:BSVP in comparison thm 3}) is reduced to
\begin{equation}\label{equation:BSVP in comparison delta=}
-4\langle y^-, f_1(s, y^++y^\prime, z,
\theta_{s+\delta(s)}^{(1)})-f_2(s, y^\prime, z^\prime,
\theta_{s+\delta(s)}^{(2)})\rangle \leq 2\sum_{k=1}^m
\mathbf{1}_{\{y_k<0\}}|z_k-z_k^\prime|^2+C|y^-|^2,\ \ {\rm{a.s.}}
\end{equation}
Note that this conclusion is just with respect to the ABSDE
(\ref{equation:peng yang comparison}) in the multidimensional case.
\end{remark}

For the special case when $f_1=f_2=:f$ and
$\delta^{(1)}=\delta^{(2)}=:\delta$, we have the following result:

\begin{theorem}
{\it The following are equivalent:

{\rm(i)} for all $\tau\in [0, T],$ $\xi^{(j)}\in
S_\mathcal{F}^2(\tau, T+K; \mathbb{R}^m)$ $(j=1, 2)$ such that
$\xi^{(1)}\geq \xi^{(2)},$ the unique solutions $(Y^{(j)},
Z^{(j)})\in S_\mathcal{F}^2(0, T+K; \mathbb{R}^m)\times
L_\mathcal{F}^2(0, \tau; \mathbb{R}^{m\times d})$ $(j=1, 2)$ to the
following ABSDE$\!:$
\begin{equation*}
\left\{
\begin{array}{llll}
-dY_t^{(j)} &\hspace{-3mm}=&\hspace{-2mm} f(t, Y_t^{(j)}, Z_t^{(j)},
Y_{t+\delta(t)}^{(j)})dt-Z_t^{(j)}dB_t,  &
t\in[0, \tau];\\
Y_t^{(j)} &\hspace{-7mm}=&\hspace{-6mm} \xi_t^{(j)},  & t\in [\tau,
T+K],
\end{array}\right.
\end{equation*}
satisfy
\begin{align*}
Y_t^{(1)}\geq Y_t^{(2)}, {\rm{\ for\ all}}\ t\in[0, \tau],\
{\rm{a.s.}}
\end{align*}

{\rm(ii)} for any $k = 1, 2,\ldots, m$, for all $s\in [0, T]$, $y\in
\mathbb{R}^m$, $\theta \in S_\mathcal{F}^2(s, T+K; \mathbb{R}^m)$,
$\theta^{(j)}\in S_\mathcal{F}^2(s, T+K; \mathbb{R}^m)$ $(j=1, 2)$
such that $\theta^{(1)} \geq \theta^{(2)}$, $f_k(s, y, \cdot,
\theta)$ depends only on $z_k$, and
\begin{equation*} f_k(s, \delta^k y+y^\prime, z_k,
\theta_{s+\delta(s)}^{(1)}) \geq f_k(s, y^\prime, z_k,
\theta_{s+\delta(s)}^{(2)}), {\rm{\ for\ any}}\ \delta^k y\in
\mathbb{R}^m {\rm{\ such\ that}}\ \delta^k y\geq 0, (\delta^k
y)_k=0.
\end{equation*}}
\end{theorem}

\begin{proof} According to Theorem 3.1, (i) is equivalent to

{(iii)} for all $s\in [0, T]$, $(y, z), (y^\prime, z^\prime)\in
\mathbb{R}^m\times \mathbb{R}^{m\times d}$ and all $\theta^{(j)}\in
S_\mathcal{F}^2(s, T+K; \mathbb{R}^m)$ $(j=1, 2)$ such that
$\theta^{(1)} \geq \theta^{(2)},$
\begin{equation}\label{equation:BSVP in comparison thm 3 f1=f2}
-4\langle y^-, f(s, y^++y^\prime, z,
\theta_{s+\delta(s)}^{(1)})-f(s, y^\prime, z^\prime,
\theta_{s+\delta(s)}^{(2)})\rangle \leq 2\sum_{k=1}^m
\mathbf{1}_{\{y_k<0\}}|z_k-z_k^\prime|^2+C|y^-|^2,\ \ {\rm{a.s.}}
\end{equation}

On the one hand, suppose that (\ref{equation:BSVP in comparison thm
3 f1=f2}) holds. Let us pick $y_k< 0$, and $y=y_ke_k$,
$\theta^{(1)}=\theta^{(2)}=:\theta$. Then we get
\begin{equation*}
4 y_k (f_k(s, y^\prime, z, \theta_{s+\delta(s)})-f_k(s, y^\prime,
z^\prime, \theta_{s+\delta(s)})) \leq 2
|z_k-z_k^\prime|^2+C|y_k|^2,\ \ {\rm{a.s.}},
\end{equation*}
which implies that $f_k$ depends only on $z_k$. Furthermore, for
$\delta^k y\in \mathbb{R}^m$ such that $\delta^k y\geq 0$,
$(\delta^k y)_k=0$, putting in (\ref{equation:BSVP in comparison thm
3 f1=f2}) $y=\delta^k y-\varepsilon e_k$, $\varepsilon >0$,
$z^\prime=z$, dividing by $-\varepsilon$ and letting $\varepsilon
\rightarrow 0^+$, we can deduce that
\begin{equation*}
f_k(s, \delta^k y+y^\prime, z_k, \theta_{s+\delta(s)}^{(1)}) \geq
f_k(s, y^\prime, z_k, \theta_{s+\delta(s)}^{(2)}).
\end{equation*}

On the other hand, it is easy to check that if (ii) holds, then
(\ref{equation:BSVP in comparison thm 3 f1=f2}) holds.
\end{proof}

\subsection{Comparison theorem in $\mathbb{R}$}

From now on, we will mainly consider the special case when $m=1.$
The following is immediate according to Remark 2.2.

\begin{theorem}
{\it Assume that $m=1$. Then the following are
equivalent$\!:$

{\rm (i)} for all $\tau\in [0, T],$ $\xi^{(j)}\in
S_\mathcal{F}^2(\tau, T+K; \mathbb{R})$ $(j=1, 2)$ such that
$\xi^{(1)}\geq \xi^{(2)},$ the unique solutions $(Y^{(j)},
Z^{(j)})\in S_\mathcal{F}^2(0, T+K; \mathbb{R})\times
L_\mathcal{F}^2(0, \tau; \mathbb{R}^{d})$ $(j=1, 2)$ to the ABSDE
{\rm(\ref{equation:conparison new tau})} with terminal conditions
$\xi^{(j)}$ $(j=1, 2)$ satisfy
$$Y_t^{(1)}\geq Y_t^{(2)}, {\rm{\ for\ all}}\ t\in[0, \tau],\ {\rm{a.s.}};$$

{\rm(ii)} for all $s\in [0, T]$, $(y, z), (y^\prime, z^\prime)\in
\mathbb{R}\times \mathbb{R}^d$ and all $\theta^{(j)} \in
S_\mathcal{F}^2(s, T+K; \mathbb{R})$ $(j=1, 2)$ such that
$\theta^{(1)} \geq \theta^{(2)},$
\begin{equation*}
f_1(s, y, z, \theta_{s+\delta^{(1)}(s)}^{(1)}) \geq f_2(s, y, z,
\theta_{s+\delta^{(2)}(s)}^{(2)}).
\end{equation*}}
\end{theorem}
\vspace{-4.5mm}

\begin{remark}
(\ref{equation:BSVP in comparison delta=}) is
equivalent to
\begin{equation}\label{equation:1 dim delta=}
f_1(s, y, z, \theta_{s+\delta(s)}^{(1)}) \geq f_2(s, y, z,
\theta_{s+\delta(s)}^{(2)}).
\end{equation}
\end{remark}

\begin{remark}
The generators $f_1$ and $f_2$ will satisfy
(\ref{equation:1 dim delta=}), if for all $s\in [0, T],$ $y\in
\mathbb{R},$ $z\in \mathbb{R}^d,$ $\theta\in L_\mathcal{F}^2(s, T+K;
\mathbb{R}),$ $r\in [s, T+K],$ $f_1(s, y, z, \theta_r)\geq f_2(s, y,
z, \theta_r),$ together with one of the following:

{(i)} for all $s\in [0, T],$ $y\in \mathbb{R},$ $z\in \mathbb{R}^d,$
$f_1(s, y, z, \cdot)$ is increasing, i.e., $f_1(s, y, z,
\theta_r)\geq f_1(s, y, z, \theta_r^\prime),$ if $\theta \geq
\theta^\prime,$ $\theta, \theta^\prime\in L_\mathcal{F}^2(s, T+K;
\mathbb{R}),$ $r\in [s, T+K];$

{(ii)} for all $s\in [0, T],$ $y\in \mathbb{R},$ $z\in
\mathbb{R}^d,$ $f_2(s, y, z, \cdot)$ is increasing, i.e., $f_2(s, y,
z, \theta_r)\geq f_2(s, y, z, \theta_r^\prime),$ if $\theta \geq
\theta^\prime,$ $\theta, \theta^\prime\in L_\mathcal{F}^2(s, T+K;
\mathbb{R}),$ $r\in [s, T+K].$

Note that the latter is just the case that Peng-Yang \cite{PY}
discussed (see Theorem 2.3).
\end{remark}

\begin{remark}
The generators $f_1$ and $f_2$ will satisfy
(\ref{equation:1 dim delta=}), if there exists a function
$\tilde{f}$ such that
\begin{equation*}
f_1(s, y, z, \theta_r)\geq \tilde{f}(s, y, z, \theta_r)\geq f_2(s,
y, z, \theta_r),
\end{equation*}
for all $s\in [0, T],$ $y\in \mathbb{R},$ $z\in \mathbb{R}^d,$
$\theta \in L_\mathcal{F}^2(s, T+K; \mathbb{R}), r\in [s, T+K].$
Here the function $\tilde{f}(s, y, z, \cdot)$ is increasing, for all
$s\in [0, T],$ $y\in \mathbb{R},$ $z\in \mathbb{R}^d,$ i.e.,
$\tilde{f}(s, y, z, \theta_r)\geq \tilde{f}(s, y, z,
\theta_r^\prime),$ if $\theta_r\geq \theta_r^\prime,$ $\theta,
\theta^\prime\in L_\mathcal{F}^2(s, T+K; \mathbb{R}), r\in [s,
T+K].$
\end{remark}

\begin{example}
Now suppose that we are facing with the following two ABSDEs:
\begin{equation*}
\left\{
\begin{array}{llll}
-dY_t^{(1)} &\hspace{-3mm}=
E^{\mathcal{F}_t}[Y_{t+\delta(t)}^{(1)}+2\sin
Y_{t+\delta(t)}^{(1)}+1]dt-Z_t^{(1)}dB_t,  & t\in[0, T];\\
Y_t^{(1)} &\hspace{-7mm}=\xi_t^{(1)},  & t\in[T, T+K],
\end{array}\right.
\end{equation*}
\begin{equation*}
\left\{
\begin{array}{llll}
-dY_t^{(2)} &\hspace{-3mm}=
E^{\mathcal{F}_t}[Y_{t+\delta(t)}^{(2)}+\cos
(2Y_{t+\delta(t)}^{(2)})-2]dt-Z_t^{(2)}dB_t,  &
t\in[0, T];\\
Y_t^{(2)} &\hspace{-7mm}= \xi_t^{(2)},  & t\in[T, T+K],
\end{array}\right.
\end{equation*}
where $\xi_t^{(1)}\geq \xi_t^{(2)}, t\in [T, T+K].$

As neither of the generators is increasing in the anticipated term
of $Y,$ we cannot apply Peng, Yang's comparison theorem to compare
$Y^{(1)}$ and $Y^{(2)}.$

While noting that $f_1(s, y, z,
\theta_r)=E^{\mathcal{F}_s}[\theta_r+2\sin \theta_r+1],$ $f_2(s, y,
z, \theta_r)=E^{\mathcal{F}_s}[\theta_r+\cos (2\theta_r)-2],$ we can
choose $\tilde{f}(s, y, z, \theta_r)=E^{\mathcal{F}_s}[\theta_r+\sin
\theta_r]$, due to the following facts:
\begin{equation*}
x+2\sin x+1\geq x+\sin x\geq x+\cos (2x)-2, x+\sin x \geq y +\sin y,
{\rm{\ for\ all}}\ x\geq y, x, y \in \mathbb{R}.
\end{equation*}
Then obviously, $f_1$, $\tilde{f}$ and $f_2$ satisfy the conditions
in Remark 3.6. Thus according to Theorem 3.3 (together with Remark
3.6), we get $Y_t^{(1)}\geq Y_t^{(2)}, {\rm{\ for\ all}}\ t\in [0,
T+K],\ {\rm{a.s.}}$
\end{example}

\begin{remark}
In \cite{PY}, Peng and Yang gave a counterexample
which indicates that if $f_2$ is not increasing in the anticipated
term of $Y$
then their comparison theorem (see Theorem 2.3) will not hold. In
fact the main reason is that the generators appearing in that
example do not satisfy the necessary and sufficient condition listed
in Remark 3.4.

That is to say, there is no contradiction between our result and
Peng and Yang's.
\end{remark}

\section*{Acknowledgements}

The author is grateful to Professor Shige Peng for his encouragement
and helpful discussions.



\end{document}